\documentclass[a4paper, 11pt]{article}

\usepackage[titletoc]{appendix}
\usepackage[T2A]{fontenc}
\usepackage[english]{babel}
\usepackage{hyperref}
\usepackage{amssymb}
\usepackage{amsmath}
\usepackage{amsthm}
\usepackage[shortlabels]{enumitem}
\usepackage{float}

\usepackage{indentfirst}

\usepackage{xcolor}

\usepackage{mathrsfs} %\mathscr

\usepackage[framemethod=TikZ]{mdframed}
\newmdenv[
    frametitle=TO-DO,
    frametitlefont=\scshape,
    backgroundcolor=lightgray, 
    roundcorner=6pt,
    leftmargin=30,
    rightmargin=30,
    linecolor=red,
    linewidth=1pt,
    skipabove=3pt,
    skipbelow=4pt,
    font=\sffamily\itshape
]{TODO}

\usepackage{orcidlink}

\usepackage{tikz}
\newtheorem{thm}{Theorem}[section]
\newtheorem{prop}[thm]{Proposition}
\newtheorem{lemm}[thm]{Lemma}

\theoremstyle{definition}
\newtheorem{defn}[thm]{Definition}

\newtheorem{exampl}[thm]{Example}

\theoremstyle{remark}

\DeclareMathOperator{\interior}{int}
\newcommand*{\closure}[1]{\overline{#1}}

\DeclareMathOperator{\distance}{dist}

\newcommand{\bR}{\mathbb{R}}

\newcommand{\bN}{\mathbb{N}}

\newcommand{\cB}{\mathcal{B}}

\newcommand{\cF}{\mathcal{F}}
\newcommand{\cG}{\mathcal{G}}

\newcommand*{\norm}[1]{\left\Vert #1\right\Vert} % norma
\newcommand*{\abs}[1]{\left\vert #1 \right\vert} % moduł | ... |
\newcommand*{\set}[2]{\left\{#1\colon #2\right\}} % zbiór {x: P(x)}
\newcommand*{\mtr}[3][X]{\left\vert #2 - #3\right\vert_{#1}}

\newcommand*{\ball}[2][r]{B\left(#2, #1\right)}
\newcommand*{\cball}[2][r]{\overline{B}\left(#2, #1\right)}

\newcommand*\diff{\mathop{}\!\mathrm{d}}

\newcommand*{\net}[3][\bN]{\left({#2}_{#3}\right)_{#3\in #1}}

\newcommand*{\cinv}[2]{\left[#1\,;#2\right]}
\newcommand*{\oinv}[2]{\left(#1\,;#2\right)}
\newcommand*{\ocinv}[2]{\left(#1\,;#2\right]}
\newcommand*{\coinv}[2]{\left[#1\,;#2\right)}

\DeclareMathOperator{\Lip}{Lip}
\DeclareMathOperator{\lip}{lip}
\newcommand*{\LLip}{\mathop{\mathbb{L}\mathrm{ip}}}

\newcommand*{\lipnorm}[1]{\left\Vert #1\right\Vert_{\mathrm{lip}}}

\newcommand*{\Neig}[1]{\mathcal{N}\left(#1\right)}

\newcommand*{\charfunction}[1]{\chi_{#1}}

\newcommand*{\lebme}{\ell}

\newcommand*{\CantorSet}{\mathbf{C}}

\title{On local Lipschitz one sets}
\author{Ziemowit M. Wójcicki}
\date{\today}

\begin{document}

\maketitle

\begin{abstract}
    We study the local Lipschitz one subsets of a finite dimensional space,
    that is, sets for which
    there exists a continuous function whose
    local Lipschitz derivative is the characteristic
    function of said set. We give a characterization
    of a local Lipschitz one set on the real line in terms
    of a certain measure-theoretic density condition, 
    which we call quasi-density.
    We show that any local Lipschitz one set 
    needs to be quasi-dense, but the converse does not hold.
    Finally, we show that any regular closed subset of a normed space
    is a local Lipschitz one set, but there exist local Lipschitz one
    sets that are not regular closed.
\end{abstract}

 \tableofcontents

\section{Introduction}

In this paper, we extend the investigation of certain questions
about the so-called little and big Lip functions, hereinafter
called Lipschitz derivatives, to the function $\LLip f$, which
in the context of our research we call the local Lipschitz 
derivative of a given function $f$.

In the setting of metric spaces devoid of linear structure,
Lipschitz derivatives seem to be a viable alternative to the traditional derivatives.
Indeed, for a differentiable function $f\colon\bR\to\bR$ it is easly seen
from the definition that $\Lip f(x)=\abs{f'(x)}$, $x\in\bR$.
In practice, Lipschitz derivatives are useful tools for investigating different properties of
differentiability or Lipschitz-like conditions of functions.
The big Lipschitz derivative is an older concept
that served investigations of classical differentiability.
For example, the classical Stepanoff's theorem \cite{Stepanoff} 
asserts that any function $f\colon\bR^n\to\bR^m$
is differentiable at almost every point of the set 
$\set{x\in\bR^n}{\Lip f(x)<\infty}$ and dates back to 1923.
On the other hand, in \cite{HerMasl},
Herasymchuk and Maslyuchenko used both big and local Lipschitz derivatives
in their investigation of separately locally Lipschitz functions.
The concept of the little Lipschitz derivative is a more recent phenomenon,
typical for nonsmooth analysis.
A notable application of the big and the little Lipschitz derivatives in this area
is the theorem by Balogh, Rogovin and Z\"{u}rcher \cite[Theorem 2.11]{BRZ}.
This theorem asserts that if
$(X,\mu)$ is a doubling metric measure space % (for the definition, see section \ref{sec:MMS}) 
and there exists a constant $K>0$ such that
for any Lipschitz function $f\colon X\to\bR$,
$\Lip f\leq K\lip f$ $\mu$-a.e.,
then $X$ admits the so-called differentiable structure.
% (see \cite{BRZ,Cheeger,Keith}).
Roughly speaking, $X$ can be decomposed into a countable collection of 
``coordinate charts'' $(X_{\alpha}, \varphi_{\alpha})$,
where $X_{\alpha}$ are measurable and the ``coordinate''
function $\varphi_{\alpha}$ is Lipschitz for any $\alpha$.
Those structures were introduced by Cheeger in \cite{Cheeger}.
See also \cite{BRZ,Keith}.
% such that for any function $f$ there exists unique measurable, 

It is natural to consider ``antiderivatives'' in 
the sense of little and big Lipschitz derivatives.
More precisely, given a function $g$, is there a 
continuous function such that $\lip f=g$ ($\Lip f=g$)?
Recently, there has been an interest in the particular case
where the function $g$ is a characteristic function of a given set.
In \cite{BHMV}, Buczolich, Hanson, Maga and V\'{e}rtesy showed that on the real line, such
sets can be characterized in terms of certain modified Lebesgue density properties.
We extend this line of investigation to the local Lipschitz derivative
and ask whether for a given subset $F$ of a metric space 
(in particular, a normed space),
there exists a continuous function $f$ such that $\LLip f$ is the characteristic
function of $F$.
In section \ref{sec:MMS}, we introduce the concept of quasi-density.
This property is simpler than those introduced in \cite{BHMV}
but can be reasonably defined on any metric measure space and
turns out to be suitable for the study of the operator $\LLip$.
We show that on the real line, local Lipschitz one sets have
a simple characterization in terms of quasi-density. Note that an analogous
condition in \cite{BHMV} characterizes \emph{monotone} big Lipschitz one
sets in terms of \emph{weak} and \emph{strong density}.
We start with the monotone case (Theorem \ref{thm:q_dense_iff_monot_LLip_one}), 
but Theorem \ref{thm:LLip_on_IR_is_q_dense} shows that
 local Lipschitz one
sets on the real line are exactly closed quasi-dense sets.

In section \ref{sec:QDRN}, we study the quasi-density in $\bR^N$.
Promisingly any local Lipschitz subset of $\bR^N$ turns out to be
a quasi-dense set. However, we give an example of a quasi-dense subset of $\bR^2$,
that is not a local Lipschitz one set.
In section \ref{sec:RCS}, we show that any regular closed subset of
a normed space is a local Lipschitz one set.  On the other hand, we 
construct a local Lipschitz one set in $\bR^2$ that is not regular closed.

\section{Lipschitz derivatives and Lipschitz one sets}

Given a metric space $X$, we denote the distance between
points $x,y\in X$ by $\mtr[X]{x}{y}$ and by
$\Neig{x}=\set{U\in 2^X}{x\in \interior U}$
the family of all neighborhoods
of a point $x$ in $X$. For $A\subseteq X$,
we denote $\distance(x,A)=\inf\limits_{a\in A}\mtr[X]{x}{a}$.

One way to define the local Lipschitz condition is to say
that a function $f\colon X\to Y$, where $X$ and $Y$ are metric
spaces, is \emph{locally Lipschitz} at a point $x\in X$ if and only if
$\LLip f(x)<\infty$, where 
\begin{equation}\label{eq:LLipDefinition}
    \LLip f(x)=\inf_{r>0}\sup_{\substack{u,v\in\ball[r]{x} \\ u\neq v}}\frac{\mtr[Y]{f(u)}{f(v)}}{\mtr[X]{u}{v}}.
\end{equation}
We call $\LLip f$ the \emph{local Lipschitz derivative of $f$}.
% $\LLip f(x)$ is the infimum of Lipschitz
% constants of $f\vert_{U}$ taken over all possible neighbors $U$ of $x$.

Given metric spaces $X$, $Y$ and a function $f\colon X\to Y$ as before,
we consider the functions $\Lip f$ and $\lip f$, called
the \emph{big} and \emph{little Lipschitz derivatives} respectively and
defined by equations
\begin{align}
   \Lip f(x) &=\limsup_{u\to x}\frac{\mtr[Y]{f(u)}{f(x)}}{\mtr[X]{u}{x}}, \label{eq:LipDefinition} \\
   \lip f(x) &=\liminf_{r\to 0+}\frac{1}{r}\sup_{u\in\ball[r]{x}}\mtr[Y]{f(u)}{f(x)}, \; \ x\in X. \label{eq:lipDefinition}
\end{align}
 %  \begin{equation}\label{eq:LipDefinition}
 %  %    \Lip f(x)=\inf_{r>0}\sup_{u\in\ball[r]{x}\setminus\{x\}}\frac{\mtr[Y]{f(u)}{f(x)}}{\mtr[X]{u}{x}}=
 %  %    \limsup_{u\to x}\frac{\mtr[Y]{f(u)}{f(x)}}{\mtr[X]{u}{x}},
 %      \Lip f(x)=    \limsup_{u\to x}\frac{\mtr[Y]{f(u)}{f(x)}}{\mtr[X]{u}{x}},
 %  \end{equation}
 %  \begin{equation}\label{eq:lipDefinition}
 %      %\lip f(x)=\liminf_{r\to 0+}\Lip^r f(x),.
 %      \lip f(x)=\liminf_{r\to 0+}\frac{1}{r}\sup_{u\in\ball[r]{x}}\mtr[Y]{f(u)}{f(x)}, \; x\in X.
 %  \end{equation}
Function $f$ is \emph{pointwise Lipschitz} at $x\in X$ if $\Lip f(x)<\infty$.
We call functions $f$ such that $\lip f(x)<\infty$ \emph{weakly pointwise Lipschitz}.
Note that $\LLip f$ can be expressed as:
\begin{equation}\label{eq:LLipAlternateDefinition}
    \LLip f(x)=\limsup_{\substack{(u,v)\to (x,x) \\ u\ne v}}\frac{\mtr[Y]{f(u)}{f(v)}}{\mtr[X]{u}{v}}, \; x\in X.
\end{equation}
It is easily seen from \eqref{eq:LipDefinition}, \eqref{eq:lipDefinition}
and \eqref{eq:LLipAlternateDefinition} that $\lip f\leq \Lip f\leq \LLip f$.
For more details on basic properties of Lipschitz derivatives, 
see for example \cite{ClassLip}.

\begin{defn}\label{def:local_Lipschitz_one_set}
We call a subset $E$ of a metric space $X$ a \emph{local Lipschitz one set}
if there exists a continuous function $f\colon X\to\bR$ such that
$\LLip f=\charfunction{E}$, where $\charfunction{E}$ stands for the characteristic
function of the set $E$.
\end{defn}

\section{Metric measure spaces and quasi-density}\label{sec:MMS}
Throughout this paper, $\lebme_N$ denotes the 
$N$-dimensional Lebesgue measure.
We will also write $\lebme$ instead of $\lebme_1$.
% $\Neigop{x}=\set{U\in\Neig{x}}{U\text{ is open}}$.

Given a metric space $X$, and a Borel regular outer measure
$\mu$ on $X$, such that $0<\mu(B)<\infty$ for any open ball
$B$ in $X$, we say that the space $(X,\mu)$ is a \emph{metric measure space}.
A metric measure space $(X,\mu)$ is called a \emph{Vitali space} (see \cite[Remark 1.13]{AnMetSp}
and \cite[Section 3.4]{SobolMMS}),
if for any set $A\subseteq X$ and family $\cF$ of closed
balls, such that for any $a\in A$ there exists $r_a>0$ such that  $\cball[r_a]{a}\in\cF$
and $\inf\set{r>0}{\cball[r]{a}\in\cF}=0$, %  holds for any $\in A$
then there exists countable disjoint subfamily $\cG$ of $\cF$ such that 
\[
    \mu\left(A\setminus\bigcup\cG\right)=0.
\]
If $\mu$ is any Radon measure on $\bR^N$, then $(\bR^N,\mu)$ is a Vitali space
(\cite[Remark 1.13]{AnMetSp}, \cite[Theorem 2.8]{Mattila}). 
Since $n$-dimensional Lebesgue measure is a Radon
measure(see \cite{Mattila} for the details) on 
$\bR^N$, $(\bR^N,\lebme_N)$ is an example of a Vitali space.
Other important example of Vitali space is $\bR^N$ equipped
with the $s$-dimensional Hausdorff measure (see \cite{Falc}). 
This is known, as the Vitali covering theorem for Hasudorff measure.
Note that Hausdorff measure is not necessarily a Radon measure
(see \cite{Mattila}).

The following version of the Lebesgue differentiation theorem holds
for a Vitali space $(X,\mu)$ (see \cite[Theorem 1.8, Remark 1.13]{AnMetSp}):

\textit{If $f\colon X\to\bR$ is a nonnegative locally integrable function,
then
\[
    \lim_{r\to 0^{+}}\frac{1}{\mu\big(\ball[r]{x}\big)}\int_{\ball[r]{x}}f(t)\diff\mu(t)=f(x),
\]
for almost every $x\in X$.}

Considering a subset $E$ of the space $(X,\mu)$ and putting
$f=\charfunction{E}$ in the above, the Lebesgue differentiation theorem yields
the following
\begin{thm}[Lebesgue density theorem]\label{thm:Leb_density_mmspace}
    Let $(X,\mu)$ be a Vitali space and $E\subseteq X$.
    Then, for almost every $x\in X$,
    \[
        d_{E}(x):=\lim_{r\to 0^{+}}\frac{\mu\big(E\cap\ball[r]{x}\big)}{\mu(\ball[r]{x})}=\charfunction{E}(x).
    \]
    In particular, 
    $\mu\Big(\big(E\setminus D(E)\big)\cup\big(D(E)\setminus E\big)\Big)=0$.
\end{thm}
%We will use Theorem \ref{thm:Leb_density_mmspace} in.
The number $d_{E}(x)$ defined above is called the \emph{Lebesgue density of $E$ at a point $x$}.
We will denote $D(E)=\set{x\in X}{d_{E}(x)=1}$.
Obviously, $D(E)\subseteq\closure{E}$.

Recall that if $\net An$ is a sequence of sets in a
given space $X$ and $x\in X$, we write $A_n\to x$ if for 
any  $U\in\Neig{x}$ there exists $n_0\in\bN$ such that 
$A_n\subseteq U$ for any $n\geq n_0$.
\begin{defn}\label{def:q_density}
    Let $(X,\mu)$ be a metric measure space.
    We say that a set $E\subseteq X$ is \emph{quasi-dense} 
    if for any $x\in E$,
    there exists a sequence $\net{A}{n}$ of subsets of $X$
    such that $A_n\to x$, $\mu(A_n)>0$, and
    \[
        \lim_{n\to\infty}\frac{\mu(E\cap A_n)}{\mu(A_n)}=1.
    \]
\end{defn}
 %  Recall, that for any $x\in\bR^N$, the 
 %  Lebesgue density of set $E\subseteq\bR$
 %  at point $x$ is defined as the number
 %  \[
 %      d_E(x)=\lim_{\varepsilon\to 0}\frac{\mu(E\cap\ball[\varepsilon]{x})}{\mu{\ball[\varepsilon]{x}}}.
 %  \]
 %  Obviously $0\leq d_E(x)\leq 1$.
 %  A point $x$ is called a \emph{point of density} of set $E$,
 %  when $d_E(x)=1$.
 %  Denote by $D(E)$ the set of all density points of set $E$. Obviously, $D(E)\subseteq\closure{E}$.

\begin{prop}\label{prop:charact_closed_qdense_in_v_space}
    Let $F$ be a closed subset of a Vitali space $(X,\mu)$. % metric measure space in which Lebesgue's theorem holds
    Then, the following conditions are equivalent:
    \begin{enumerate}[label=\normalfont{(\roman*)}]
        \item $F$ is quasi-dense,
        \item for any $x\in F$ and open $U\in\Neig{x}$
              we have $\mu(F\cap U)>0$,
        \item $\closure{D(F)}=F$,
        \item for any $x\in F$ there exists a sequence $\net{B}{n}$ of open balls in $X$
              such that ${B_n\to x}$ and
              $
                  \lim\limits_{n\to\infty}\frac{\mu(F\cap B_n)}{\mu(B_n)}=1.
              $
    \end{enumerate}
\end{prop}
\begin{proof}
    (i)$\Rightarrow$(ii) Assume that for some $x\in F$
    and open $U\in\Neig{x}$ we would have $\mu(F\cap U)=0$.
    Choose a sequence
    $\net{A}{n}$ as in definition \ref{def:q_density}.
    Then, for some $n_0\in\bN$, $A_{n}\subseteq U$ whenever $n\geq n_0$.
    But then
    \[
        0=\lim\limits_{n\to\infty}\frac{\mu(U\cap F)}{\mu(A_n)}\geq
        \lim\limits_{n\to\infty}\frac{\mu(A_n\cap F)}{\mu(A_n)}=1,
    \]
    which is an obvious contradiction.

    (ii)$\Rightarrow$(iii) Consider $x\in F$ and an 
    open neighborhood $U$ of $x$ in $X$.
    We have $\mu(U\cap F)>0$. But Theorem \ref{thm:Leb_density_mmspace}
    implies that $\mu(F\setminus D(F))=0$, so necessarily
    $U\cap F\nsubseteq F\setminus D(F)$. Then, there exists $y\in U\cap F$
    such that $y\notin F\setminus D(F)=F\cap(X\setminus D(F))$.
    Thus, $y\notin X\setminus D(F)$ or, equivalently $y\in D(F)$.
    So, $y\in U\cap D(F)$. Thus, $U\cap D(F)\neq\varnothing$ for any open $U\in \Neig{x}$. Hence,
     $x\in\closure{D(F)}$.
    Therefore, $F\subseteq\closure{D(F)}$.
    
    Now, since $F$ is closed, we also have 
    $\closure{D(F)}\subseteq\closure{F}=F$.
%   Now fix $x\in\closure{D(F)}$ and some $U\in\Neig{x}$.
%   Then, 
%   $\mu(U\cap D(F))\geq\mu(U\cap F)>0$,
%   since $F\subseteq\closure{D(F)}$, as we have already shown.
%   So $U\cap D(F)\nsubseteq D(F)\setminus F$ and there exists
%   $y\in U\cap D(F)$ such that $y\notin D(F)\setminus F$.
%   We have $y\in U\cap D(F)\cap F\subseteq U\cap F$, so
%   $U\cap F\neq\varnothing$. This proves, that $x\in\closure{F}=F$.
%   %since $F$ is closed set.

    (iii)$\Rightarrow$(iv)
    Fix $x\in F$. 
    Since $\closure{D(F)}=F$,
    there exists a sequence $(a_n)_{n=1}^{\infty}$
    in $D(F)$ such that $\mtr[X]{a_n}{x}<\frac{1}{n}$, $n\in\bN$.
    So, for any $n\in\bN$, we have that
    \[
        \lim_{r\to 0^{+}}\frac{\mu(F\cap\ball[r]{a_n})}{\mu(\ball[r]{a_n})}=\charfunction{F}(a_n)=1.
    \]
    Thus, there is $r_n\in\oinv{0}{\frac{1}{n}}$ such that
    \[
        1-\frac{1}{n} < \frac{\mu(F\cap B_n)}{\mu(B_n)} \leq 1,
    \]
    where $B_n=\ball[r_n]{a_n}$.
    So $\frac{\mu(F\cap B_n)}{\mu(B_n)} \to 1$. On the other hand, if $y\in B_n$,
    then %$\mtr[X]{y}{x}\leq\mtr[X]{y}{a_n}+\mtr[X]{a_n}{x} < \frac{2}{n}$.
    \[\mtr[X]{y}{x}\leq\mtr[X]{y}{a_n}+\mtr[X]{a_n}{x} < \frac{2}{n}.\]
    Therefore $B_n\to x$.

    (iv)$\Rightarrow$(i) obvious.
\end{proof}

\begin{exampl}\label{ex:FCS}
    We construct a symmetric fat Cantor set
    and demonstrate that it is quasi-dense.
    We base our construction on the one 
    approach given in \cite[p. 39]{Folland}. 
    Let $I_1=\cinv{0}{1}$ and let
    $\alpha=(\alpha_n)_{n=1}^{\infty}$
    be a sequence of positive real numbers
    such that $\sum_{n=1}^{\infty}2^{n-1}\alpha_n<1$.
    Assume that we have $2^{n-1}$ disjoint closed intervals 
    $I_{2^{n-1}}$, $\ldots$, $I_{2^{n}-1}$ of equal length.
    We create $2^{n}$ intervals
    $I_{2^n}$, $\ldots$, $I_{2^{n+1}-1}$:
    for each $k=0,\ldots,2^{n-1}-1$,
    by removing from the interval 
    $I_{2^{n-1}+k}$,
    the open subinterval of the length $\alpha_n$, 
    centered at the middle of the interval
    $I_{2^{n-1}+k}$.
    Let $C_n=I_{2^n}\cup I_{2^n+1}\cup\ldots\cup I_{2^n+2^{n-1}-1}$
%  \[
%      C_n=\bigcup_{k=0}^{2^n-1}I_{2^n+k}, \text{ for any }n.
%  \]
    We define the fat Cantor set to be the intersection of sets $C_n$:
    $
        \CantorSet_{\alpha} = 
        \bigcap_{n=0}^{\infty}C_n.
    $
    One can readily verify that
    $\lebme(\CantorSet_{\alpha}) = 1 - \sum_{n=1}^{\infty}2^{n-1}\alpha_n>0$.
    
    It remains to show that $\CantorSet_{\alpha}$ is quasi-dense.
    Observe that  $\lebme(\CantorSet_{\alpha}\cap I_{2^{n}+k})=\lebme(\CantorSet_{\alpha}\cap I_{2^{n}+l})$ for any
    $n\in\bN$ and $k, l\in\{0,\ldots,2^{n-1}-1\}$. Therefore
     \[
     \lebme(\CantorSet_{\alpha}\cap I_{2^{n}+k})=\frac{1}{2^n}\lebme(\CantorSet_{\alpha}), \text{ for any } n\in\bN, k=0,\ldots,2^{n-1}.
     \]
     Fix $x\in\CantorSet_{\alpha}$ and $U=\oinv{x-\varepsilon}{x+\varepsilon}$,
     $\varepsilon>0$. Choose $n$ such that
     $\frac{1}{2^n}<\varepsilon$ and $k$ such that
     $x\in I_{2^n+k}$. Since $I_{2^n+k}\subseteq U$, we have
     \[
        \lebme(\CantorSet_{\alpha}\cap U)\geq\lebme(\CantorSet_{\alpha}\cap I_{2^{n}+k})=\frac{1}{2^n}\lebme(\CantorSet_{\alpha})>0.
     \]
     Thus, $\CantorSet_{\alpha}$ is quasi-dense by Proposition \ref{prop:charact_closed_qdense_in_v_space}. %(ii).
     \qed
\end{exampl}

In fact, it is sufficient
to have a set of positive measure 
in the Vitali space in order to find
a quasi-dense set, as shown
in the following % Proposition.
\begin{prop}\label{prop:qdense_subset_of_posm_set}
    Let $A$ be a closed subset of a Vitali space $(X,\mu)$.
    If $\mu(A)>0$, then there exists a closed quasi-dense
    subset $B$ of the set $A$ such that $\mu(A\setminus B)=0$.
\end{prop}
\begin{proof}
    Let $G=\bigcup\set{U}{U\text{ is open and }\mu(A\cap U)=0}$ and let $B=A\setminus G$. Then it is easy to see
    that $B$ is a quasi-dense subset of $A$.
\end{proof}

\begin{exampl}
    We now show an example of a set that is nowhere dense and
    simultaneously quasi-dense in more general metric space.
    
    Let $(X,\mu)$ be a separable Vitali space with $\mu$
    non-atomic and $B=\cball[r]{a}$ for some $a\in X$ and $r>0$.
    Moreover, let $\set{x_n\in X}{n\in\bN}$ be a dense
    subset of $B$.
    Since $\mu$ is non-atomic, $\lim\limits_{r\to 0}\mu\big(\cball[r]{x_n}\big)=0$ for any $n\in\bN$. 
    Hence there exists $\rho_n>0$ such that
    \[
        \mu\big(\cball[\rho_n]{x_n}\big) < \frac{\mu(B)}{3^n}.
    \]
    Then, the set
    \[
        E=B\setminus\bigcup_{n=1}^{\infty}\ball[\rho_n]{x_n}
    \]
    is a closed nowhere dense set.
    Moreover, 
    \begin{align*}
        \mu(E) 
        &= \mu(B)-\mu\left(\bigcup_{n=1}^{\infty}\cball[\rho_n]{x_n}\cap B\right) \\
        &\geq\mu(B)-\sum_{n=1}^{\infty}\mu\big(\cball[\rho_n]{x_n}\big) \\
        &\geq \mu(B)\left(1-\sum_{1}^{\infty}\frac{1}{3^n}\right) \\
        &=\mu(B)\left(1-\frac{1}{2}\right) = \frac{1}{2}\mu(B)>0.
    \end{align*}
    By Proposition \ref{prop:qdense_subset_of_posm_set} there
    exists a quasi-dense closed subset $F$ of set $E$.
    \qed
\end{exampl}

%   \begin{exampl}
%       Let $I=\cinv{0}{1}$ and let $\set{x_n}{n\in\bN}$
%       be a dense subset of $I$. We define $G$ to be a set
%       \begin{TODO}
%           Finish.
%       \end{TODO}
%       Let $F=I\setminus \bigcup G$.
%   \end{exampl}

\section{Local Lip one sets on the real line}

\begin{thm}\label{thm:q_dense_iff_monot_LLip_one}
    Let $F$ be a closed subset of $\bR$. Then,
    $F$ is quasi-dense, if and only if there exists
    a monotone continuous function $f\colon\bR\to\bR$
    such that $\LLip f=\charfunction{F}$.
\end{thm}
\begin{proof}
    The ``if'' part will be proven in Theorem \ref{thm:LLip_is_q_dense}
    with greater generality.

    For the ``only if`` part, assume, that $F$ is quasi-dense.
    Fix some $a\in\bR$ and
    let $f(x)=\lebme(\cinv{a}{x}\cap F)$, $x\in \bR$.
    Note that $\abs{f(u)-f(v)}=\lebme(\ocinv{u}{v}\cap F)$, $u,v\in\bR$, $u\leq v$
    and the function $f$ is obviously monotone.
    Fix $x\notin F$. Then $\distance(x,F)>0$ and for
    $0<r<\distance(x,F)$ we have $\ball[r]{x}\cap F=\varnothing$.
    % we have $F\cap\ocinv{u}{v}=\varnothing$
    %for any $u,v\in\ball[r]{x}=\oinv{x-r}{x+r}$, $u<v$.
    Hence,
    \begin{align*}
        \LLip f(x)&=\inf_{r>0}\sup_{\substack{u,v\in\ball[r]{x} \\ u<v}}\frac{\abs{f(u)-f(v)}}{\abs{u-v}} \\
        &=\inf_{r>0}\sup_{\substack{u,v\in\ball[r]{x} \\ u<v}}\frac{\lebme\big(F\cap\ocinv{u}{v}\big)}{\lebme\big(\ocinv{u}{v}\big)} = 0.
%       &=\inf_{r>0}\sup_{\substack{u,v\in\ball[r]{x} \\ u<v}}\frac{\lebme(F\cap\cinv{u}{v})}{\lebme(\cinv{u}{v})}=0.
    \end{align*}
    
    Consider $x\in F$. 
    For any $u<v$,
    \[
        \frac{\abs{f(u)-f(v)}}{\abs{u-v}}=\frac{\lebme\big(F\cap\ocinv{u}{v}\big)}{\lebme(\ocinv{u}{v})}\leq 1.
    \]
 %  For any sequence $(u_n,v_n)_{n=1}^{\infty}$ of
 %  numbers, where $u_n< v_n$, $n\in\bN$, we necessarily have
 %  $\frac{\lebme(F\cap\cinv{u_n}{v_n})}{\lebme(\cinv{u_n}{v_n})}\leq 1$.
 %  So, if there exist a limit of a sequence 
 %  $\frac{\abs{f(u_n)-f(v_n)}}{\abs{u_n-v_n}}$,
 %  then
 %  \[
 %      \lim_{n\to\infty}\frac{\abs{f(u_n)-f(v_n)}}{\abs{u_n-v_n}}=\lim_{n\to\infty}\frac{\lebme(F\cap\cinv{u_n}{v_n})}{\lebme(\cinv{u_n}{v_n})}\leq 1.
 %  \]
 %  On the other hand, by quasi-density of $F$, there exists such
 %  sequence of intervals $I_n=\cinv{u_n}{v_n}$, that $I_n\to x$ and
 %  \[
 %      \lim_{n\to\infty}\frac{\abs{f(u_n)-f(v_n)}}{\abs{u_n-v_n}}=
 %      \lim_{n\to\infty}\frac{\lebme(F\cap\cinv{u_n}{v_n})}{\lebme(\cinv{u_n}{v_n})}=
 %      \lim_{n\to\infty}\frac{\lebme(F\cap I_n)}{\lebme(I_n)}=1.
 %  \]
 %  It follows, that
    On the other hand, by quasi-density of $F$ and the condition
    (iv) of Proposition~\ref{prop:charact_closed_qdense_in_v_space}, 
    there exists sequence of open balls $B_n=\oinv{u_n}{v_n}$
    such that $B_n\to x$ and
    $\frac{\lebme\big(F\cap B_n\big)}{\lebme(B_n)}\to 1$.
    It follows, that
    \begin{align*}
        1\geq \LLip f(x) &= \limsup_{(u,v)\to(x,x)}\frac{\abs{f(u)-f(v)}}{\abs{u-v}}
        \geq \lim_{n\to\infty}\frac{\abs{f(u_n)-f(v_n)}}{\abs{u_n-v_n}} \\
        &= \lim_{n\to\infty} \frac{\lebme\big(F\cap B_n\big)}{\lebme(B_n)} = 1.\qedhere
    \end{align*}
 %  Thus, we have shown, that $\LLip f=\charfunction{F}$.
\end{proof}

\begin{thm}\label{thm:LLip_on_IR_is_q_dense}
    Let $F$ be a closed subset of $\bR$ such that there exists
    a continuous function $f$ satisfying $\LLip f=\charfunction{F}$. 
    Then $F$ is quasi-dense.
\end{thm}
\begin{proof}
    Assume that $F$ is not quasi-dense.
    By Proposition \ref{prop:charact_closed_qdense_in_v_space},
    there exist $x_0\in F$ and $U=\oinv{a}{b}\subseteq\bR$
    with $x_0\in U$ such that ${\lebme(F\cap U)=0}$.
    Let $f\colon\bR\to\bR$ be such that $\LLip f=\charfunction{F}$.
    By \cite[Theorem 6.4]{ClassLip}, $f$ is $1$-Lipschitz.
    In particular $f$ is absolutely continuous and is therefore differentiable a.e. 
    Then 
    $\abs{f'(x)}=\Lip f(x)\leq\LLip f(x)=\charfunction{F}(x)$
    for a.e. $x\in\bR$.
    Since $\lebme(F\cap U)=0$, we have $\charfunction{F}(x)=0$ for a.e. $x\in U$.
    Hence, $\abs{f'(x)}=f'(x)=0$ for a.e. $x\in U$. Since $f$ is absolutely continuous, we have
    \[
        f(x)=f(x_0)+\int_{x_0}^{x}f'(t)\diff t = f(x_0) + 0 = f(x_0),
    \]
    for any $x\in U$. That is, $f$ is constant on $U$.
    Then $\charfunction{F}(x)=\LLip f(x)=0$ for any $x\in U$, and so $F\cap U=\varnothing$.
    A contradiction, since $x_0\in F\cap U$.
\end{proof}

\section{Quasi-density in $\bR^N$}\label{sec:QDRN}

%   \begin{thm}\label{thm:small_lip_and_gamma_lip}%[{\cite[Theorem 6.4.]{ClassLip}}]
%       Let $D$ be a convex subset of $\bR^N$,
%       $f\colon D \to \bR$ be a function and let $\gamma\geq 0$. 
%       Then $f$ is $\gamma$-Lipschitz if and only
%       if $\lip f(x) \leq \gamma$ for any $x \in D$.
%   \end{thm}
%   Theorem \ref{thm:small_lip_and_gamma_lip}
%   is an immediate consequence of \cite[Theorem 6.4.]{ClassLip}.

\begin{prop}\label{prop:LLip_zero_implies_const}
    Let $G$ be an open and connected subset of $\bR^N$
    and let $f\colon G\to\bR$ be a Lipschitz function
    such that $\LLip f(x)=0$ for a.e. $x\in G$. Then $f$
    is constant on $G$.
\end{prop}
\begin{proof}
    We start with the case $N=1$. 
    We can assume, that $G=\oinv{a}{b}$
    for some reals $a<b$.
    Since $f$ is absolutely continuous, by Lebesgue's theorem $f$ is differentiable a.e.
    and for any $a < x\leq y < b$,
    \[
       \abs{f(x)-f(y)}\leq\int_{x}^{y}\abs{f'(t)}\diff\lebme( t)=\int_{x}^{y}\Lip f(t)\diff\lebme( t) \leq \int_{x}^{y}\LLip f(t)\diff\lebme( t)= 0,
    \]
    so, $f(x)=f(y)$.

    Now, assume that the proposition was proven for some positive integer $N-1$.
    We will show, that the thesis holds for $N$.
    First, consider the case where $G=\prod\limits_{k=1}^{N}\oinv{a_k}{b_k}$.
    Denote $X=\prod\limits_{k=1}^{N-1}\oinv{a_k}{b_k}$ and $Y=\oinv{a_N}{b_N}$,
    so $G=X\times Y$.
    By the assumption, there exists a set $E\subseteq G$ such that $\lebme_{N}(E)=0$
    and $\LLip f(x)=0$ for $x\in G\setminus E$. 
%   Set $G$ is of a form $X\times Y$, where $X=\prod\limits_{k=1}^{N-1}\oinv{a_k}{b_k}$
%   and $Y=\oinv{a_N}{b_N}$.
    By the Fubini theorem
    \begin{align*}
        0 = \lebme_{N}(E) &= \int_{G}\charfunction{E}(t)\diff\lebme_{N}(t) 
        = \int_{X}\left(\int_{Y}\charfunction{E}(x,y)\diff\lebme_{1}(y)\right)\diff\lebme_{N-1}(x) \\
        &= \int_{X}\lebme_{1}(E^{x})\diff\lebme_{N-1}(x),
    \end{align*}
    where $E^{x}=\set{y\in Y}{(x,y)\in X\times Y}$.
    This means, that there exists a set $A\subseteq X$ of measure zero
    such that $\lebme_{1}(E^{x})=0$ for $x\in A$.
    Fix $x\in X\setminus A$.
    We can consider a function $f^{x}\colon Y\to\bR$ defined by
    \[
        f^{x}(y)=f(x,y) \text{ for } y\in Y.
    \]
    For any $y\in Y\setminus E^{x}$, $\LLip f^{x}(y)=0$, so
    $\LLip f^{x}(y)=0$ for a.e. $y\in Y$.
    By the previously proven case for $N=1$, the function $f^{x}$
    is constant on $Y$. Thus,
    \[
        f(x,y_1)=f(x,y_2), \; \text{ for any } x\in X\setminus A, \; y_1,y_2\in Y
    \]
    On the other hand, since $A\subseteq X$ is a null set,
    $X\setminus A$ is dense in $X$. Hence,
    $f(x,y_1)=f(x,y_2)$ for any $y_1,y_2\in Y$ and any $x$ in $X$.
    We have proven that $f$ is constant with respect to the last
    variable. Analogous reasoning shows that $f$ is constant
    with respect to any variable. Thus, $f$ is constant on $G$.

    Now consider an arbitrary open connected set $G$.
    Let $a\in G$ and denote
    $H=\set{x\in G}{f(x)=f(a)}$.
    Then, $H$ is closed in $G$.
    To prove that $H$ is open, consider some point
    $x\in G$. Then there exists a neighborhood
    $U=\prod\limits_{k=1}^{N}\oinv{a_k}{b_k}\subseteq G$
    of $x$. Then $f$ is constant on $U$ by the previous case.
    So, $f(u)=f(x)=f(a)$ for any $u\in U$. Therefore, $U\subseteq H$.
    Thus $H$ is closed and open nonempty set in the connected
    set $G$. This implies $H=G$.
\end{proof}

\begin{thm}\label{thm:LLip_is_q_dense}
    If $F$ is a local Lipschitz one set in $\bR^N$
    then $F$ is quasi-dense.
\end{thm}
\begin{proof}
    Assume that $F$ is not quasi-dense. Then, by Proposition
    \ref{prop:charact_closed_qdense_in_v_space}$(ii)$, there exists a point $x_0\in F$
    and an open set $U\subseteq\bR^N$ such that $\lebme_{N}(U\cap F)=0$.
    We can assume that $U=\ball[r]{x_0}$ for some $r>0$.
    Then $\LLip f(x)=\charfunction{F}(x)=0$ for any $x\in U\setminus F$.
    Consequentially, $\LLip f(x)=0$ a.e.~on $U$.
    Thus, by Proposition \ref{prop:LLip_zero_implies_const}, $f$ is constant on $U$.
    %so $\LLip f(x)=0$ for any $x\in U$.
    Hence, $\charfunction{F}(x)=0$ for any $x\in U$.
    But this means, that $U\cap F=\varnothing$ -- a contradiction.
\end{proof}

 %  \begin{proof}[Second version of the proof of Theorem \ref{thm:LLip_is_q_dense}]
 %      For $N=1$ this was already proven as Theorem \ref{thm:LLip_on_IR_is_q_dense}.
 %      Assume, that thesis is true for some positive integer $N-1$.
 %      Let $F\subseteq\bR^{N}$ be a $\LLip$ one set.
 %      Assume, that $F$ is not quasi-dense. Then, by Proposition
 %      \ref{prop:charact_closed_qdense_in_v_space}$(ii)$, there exists an
 %      open set $U\subseteq\bR^N$ such that $\lebme_{N}(U\cap F)=0$. 
 %      Denote $E=U\cap F$ and $X=\projection_{\bR^{N-1}}(E)$.
 %      Then, by Fubini's theorem,
 %      \[
 %          0=\lebme(E)=\int_{E}\diff\lebme_{N}=\int_{X}\lebme(E^{x})\diff\lebme(x).
 %      \]
 %      So, there exists a subset $A$ of $\bR^{N-1}$ of measure zero, such that
 %      $\lebme(E^{x})=0$ for any $x\in X\setminus A$.
 %      Fix $x\in X\setminus A$. Let $f\colon\bR^N\to\bR$ be a function such that
 %      $\LLip f=\charfunction{F}$.
 %  \end{proof}

%  \begin{prop}
%      Let $f\colon\bR^2\to\bR$ be a $1$-Lipschitz function, such that
%      $\dfrac{\partial f}{\partial x}=\dfrac{\partial f}{\partial y}=0$
%      almost everywhere on $\bR^2$. Then $f$ is a constant function.
%  \end{prop}
%  

Quasi-density turns out to be insufficient even
for a subset of $\bR^2$ to be a local Lipschitz one set.
To show this, we need the following
\begin{prop}\label{prop:LLip_one_disconnected}
    Let $F$ be a nonempty nowhere dense local Lipschitz one subset of $\bR^2$.
    Then $\bR^2\setminus F$ is a disconnected set.
\end{prop}
\begin{proof}
    Denote by $G$ the open set $\bR^2\setminus F$. By the definition $\LLip f = 0$ on $G$
    and hence, $f$ is locally constant on $G$. 
    Suppose that $G$ is connected. Then, $f\vert_{G}$ is constant.
    Since $F$ is nowhere dense in $\bR^2$, then $\closure{G}=\bR^2$.
    Hence $f=f\vert_{G}$ on $\bR^2$. This means that $f$ is constant 
    and so, $\LLip f = 0$ on $\bR^2$, which is a contradiction.
\end{proof}

\begin{exampl}
    We give an example of a quasi-dense subset of 
    $\bR^2$ that is not a local Lipschitz one set.
    
    Let $\mathbf{C}$ be a symmetric fat Cantor set,
    as in Example \ref{ex:FCS}.
    Then, it is easy to see that $\mathbf{C}^2$ is quasi-dense. 
    However, $\bR^2\setminus\mathbf{C}^2$ is a
    connected set, so $\mathbf{C}^2$ cannot be a local Lipschitz one
    set by Proposition \ref{prop:LLip_one_disconnected}.
    \qed
\end{exampl}
%Put $F=\mathbf{C}^2$. Then $\bR^2\setminus F$ is a connected set,
%so $F$ cannot be a local Lipschitz one set.
%$\set{a\in\bR^2}{\norm{a}\in\mathbf{C}}$.
%   We will show, that $\bR^2\setminus F$ is a connected set.
%   Fix $p_1=(x_1,y_1)$, $p_2=(x_2,y_2)\in\bR^2\setminus F$
%   such that $p_1\neq p_2$.

\section{Regular closed sets are local Lipschitz one sets}\label{sec:RCS}

Recall, that a subset $F$ of a metric space $X$ is called regular closed if
$F=\closure{\interior F}$.
The main result of this note is given by the following
\begin{thm}\label{thm:regular_closed_lip_set}
    Every regular closed subset of a separable normed space is a local Lipschitz one set.

   % Let $X$ be a separable normed space. and let $F\subseteq X$ be a
  %  regular closed set.
   % Then,  % there exists a continuous function $f\colon F\to\bR$ 
    % such that $\LLip f=\charfunction{F}$.
\end{thm}

We will split the main steps of the proof into lemmas, starting with the following
\begin{lemm}\label{lem:U_n_Lipschitz}
    Let $(U_n)_{n=1}^{\infty}$ be a sequence of open convex
    pairwise disjoint subsets of a normed space $X$ and
    let $F=\closure{\bigcup_{n=1}^{\infty}U_n}$.
    Consider the sequence of $1$-Lipschitz functions $f_n\colon X\to\bR$, such that
    $
        f_n=0 \text{ on } X\setminus U_n,
    $
    and $\LLip f_n=1$ on $U_n$ for any $n\in\bN$.
    Then the function $f\colon X\to\bR$ defined by the formula
    $f=\sum\limits_{n=1}^{\infty}f_n$ is a $1$-Lipschitz function,
    such that $\LLip f=\charfunction{F}$.
\end{lemm}
\begin{proof}
    %\sloppy
    We will show that $f$ is $1$-Lipschitz.
    For this purpose fix $x,y \in X$. 
    Let $G=\bigcup_{n=1}^{\infty}U_n$. We will consider three cases:

    \textbf{Case 1:} $x,y\notin G$. 
    In this case, $f(x)=f(y)=0$, so ${\abs{f(x)-f(y)}\leq\norm{x-y}}$.
    
    \textbf{Case 2:} $x\in G$, $y\notin G$.
    There is some $n$, such that $x\in U_n$, and so, $f(x)=f_n(x)$.
    On the other hand $f(y)=0=f_n(y)$. 
    Then,
        $\abs{f(x)-f(y)}=\abs{f_n(x)-f_n(y)}\leq\norm{x-y}$,
  % \[
  %     \abs{f(x)-f(y)}=\abs{f_n(x)-f_n(y)}\leq\norm{x-y},
  % \]
    since $f_n$ is $1$-Lipschitz.
        
    \textbf{Case 3:} $x,y\in G$. 
    We have $x\in U_n$ and $y\in U_m$ for some positive integers $n$ and $m$.
    Let us consider further two cases.
    
    $1^\circ$ If $n=m$, then $\abs{f(x)-f(y)}=\abs{f_n(x)-f_n(y)}\leq\norm{x-y}$.
    
    $2^\circ$ If $n\neq m$, then by the assumption, $U_n\cap U_m=\varnothing$. Let
    $\varphi(t)=(1-t)x+ty$, $t\in\cinv{0}{1}$ and $L=\varphi\big[\cinv{0}{1}\big]$.
    Since $\varphi$ is affine, the set $I_k=\varphi^{-1}[U_k]$ is
    convex for any $k\in\bN$. The sets $I_n$ and $I_m$ are disjoint and open
    in $\cinv{0}{1}$. Therefore, for some $0<s_0\leq t_0<1$, $I_n=\coinv{0}{s_0}$
    and $I_m=\ocinv{t_0}{1}$. Let $z=\varphi(t_0)$. We have
    $z\in L\setminus(U_n\cup U_m)$, since $t_0\notin I_n\cup I_m$.
    This implies that $f_n(z)=f_m(z)=0$.
    %$L=\set{(1-t)x+ty}{t\in\cinv{0}{1}}$.
    We can now estimate:
    \begin{align*}
        \abs{f(x)-f(y)} &= \abs{f_n(x)-f_m(y)} \\
        & \leq \abs{f_n(x)-f_n(z)}+\abs{f_m(z)-f_m(y)} \\
        & \leq \norm{x-z} + \norm{z-y} = \norm{x-y}.
    \end{align*}
    We have shown that $f$ is a $1$-Lipschitz function and as such,
    ${\LLip f(x)\leq 1}$, $x\in X$.
    Note that $\LLip f(x)=0$ for $x\in X\setminus F$, because $f(x)=0$ for such $x$.
    Consider $x\in G$. Then there exists $n\in\bN$ for which $x\in U_n$.
    We have
    $
        \LLip f(x) = \LLip f_n(x)=1.
    $
    %and, since $U_n$ is neighborhood of $x$...
    Let $E=\set{x\in X}{\LLip f(x) \geq 1}$.
    Obviously, the set $E$ is closed, since $\LLip f$ is upper-semicontinuous (see \cite[Theorem 5.5]{ClassLip}).
    On the other hand, $E=\set{x\in X}{\LLip f(x)=1}$, since $\LLip f(x)\leq 1$.
    By the definition of $E$ and properties of $f$, we have the inclusions
    $G\subseteq E\subseteq F$.
    Hence, $\closure{G}\subseteq E\subseteq F$. But $\closure{G}=F$ by the definition.
    Thus, $F=E$, and we have shown that $\LLip f=\charfunction{F}$.
\end{proof}

The following topological lemma follows for example from
\cite[Lemat 1.4.8]{Blaszcz}.
\begin{lemm}\label{lem:countable_disjoint_family}
    Let $X$ be a topological space, let $\cB$ be a countable
    base of $X$ and let $F$ be a regular closed subset of $X$.
    Then there exists a disjoint family $\cB_0\subseteq\cB$
    such that $F=\closure{\bigcup\cB_0}$.
\end{lemm}

\begin{lemm}\label{lem:fmax_Lip_ball}
    Let $X$ be a normed space and $U=\ball[\varepsilon]{a}$.
    Define a function $f$ by setting
    $f(x)=\max\{0, \varepsilon-\norm{x-a}\}$ for any $x\in X$.
    Then $f$ is a $1$-Lipschitz function and $\LLip f=\charfunction{\closure{U}}$.
\end{lemm}
\begin{proof}
    Since the norm map $x\mapsto\norm{x}$, $x\in X$ is a $1$-Lipschitz function,
    $f$ is also $1$-Lipschitz.
    For $x\in X\setminus\closure{U}$ we have $f(x)=0$, so $\LLip f(x)=0$.
    Fix $x\in\closure{U}=\cball[\varepsilon]{a}$ and an arbitrary $r>0$.
    Consider two cases.
    
    \textbf{Case 1:} $x=a$. Let $y\in U\setminus\{a\}$, such that $\norm{x-y}<r$.
        We have
        \begin{align*}
            \abs{f(x)-f(y)} &= \abs{f(a)-f(y)} = \abs{\varepsilon - (\varepsilon - \norm{a-y})} \\
            &= \norm{y-a} = \norm{x-y}.
        \end{align*}

    \textbf{Case 2:} $x\neq a$.
      % we can choose $y$, so
      % $y=tx+(1-t)a$ for some $t\in\oinv{0}{1}$.
      % Then, $y\in U$. For any $r>0$, 
      % $\norm{y-a}=\norm{tx+a-ta-a}=\abs{t}\norm{x-a}$,
        Fix $t$ such that $1-\frac{r}{\norm{x-a}}<t<1$.
        Then, the point $y=tx+(1-t)a$ lies inside the ball $\ball[\varepsilon]{a}$.
        Note that, $\norm{y-a}=t\norm{x-a}$. %$ < \varepsilon$.
        % Note that $\norm{y-a}=\norm{xt+(1-t)a-a}=t\norm{x-a}$.
        We have
        \begin{align*}
            \norm{x-y} &= \norm{x-tx -a +ta} = \norm{x-a+t(a-x)} \\
            &= \norm{(x-a)-t(x-a)} = \norm{(x-a)(1-t)} \\
            &= \abs{1-t}\norm{x-a} < r.
        \end{align*}
        So $y\in\ball[r]{x}$. On the other hand
        \begin{align*}
            \abs{f(x)-f(y)} &= \abs{\varepsilon - \norm{x-a} - (\varepsilon - \norm{y-a})} \\
                &= \abs{\norm{x-a}-\norm{y-a}} = \abs{\norm{x-a}-t\norm{x-a}} \\
                &= \abs{1-t}\norm{x-a} = \norm{x-y}.
        \end{align*}       
        
    In both cases, $\displaystyle\frac{\abs{f(x)-f(y)}}{\norm{x-y}} = 1$ for some $y\in\ball[r]{x}$.
    Thus, ${\LLip f(x)\geq 1}$ for $x\in\closure{U}$. But, since $f$ is $1$-Lipschitz,
    we have $\LLip f(x)=1$ for $x\in\closure{U}$.
\end{proof}

\begin{proof}[Proof of Theorem \ref{thm:regular_closed_lip_set}]
%\textit{Proof ot Theorem \ref{thm:regular_closed_lip_set}.}
Let $X$ be a separable normed space and $F$ be
a regular closed subset of $X$. Then there exists a dense countable
subset $A$ of $X$. Let $\cB=\set{\ball[\frac{1}{k}]{a}}{a\in A,k\in\bN}$.
The family $\cB$ is a countable basis of $X$.
By Lemma \ref{lem:countable_disjoint_family}
we can choose a disjoint subfamily $\cB_0$ of $\cB$ such that
$F=\closure{\bigcup\cB_0}$. We can assume that
$\cB_0=\set{U_n}{n\in\bN}=\set{\ball[\varepsilon_n]{a_n}}{n\in\bN}$,
where $a_n\in A$ for each $n\in\bN$.
Fix $n\in\bN$. Let $f_n(x)=\max\{0,\varepsilon_n-\norm{x-a_n}\}$.
By Lemma \ref{lem:fmax_Lip_ball}, $f_n$ is $1$-Lipschitz function
such that $\LLip f_n=\charfunction{\closure{U_n}}$.
Hence, $\LLip f_n(x)=1$ for any $x\in U_n$.
On the other hand, by the definition of $f_n$, we have
$f_n(x)=0$ for $x\in X\setminus U_n$.
So, by Lemma \ref{lem:U_n_Lipschitz}, the function
$f=\sum_{n=1}^{\infty}f_n$ satisfies
$\LLip f=\charfunction{F}$.
\end{proof}

\begin{exampl}
    We will end with an example of a local Lipschitz one subset
    of the plane, that is not a regular closed set.
    
    Let $\mathbf{C}$ be a symmetric fat Cantor set.
    Put $$F=\set{a\in\bR^2}{\norm{a}\in\mathbf{C}}.$$
    %Then, set $F$ is a $\LLip$ one subset of $\bR^2$.
    
    Since the mapping $\varphi(a)=\norm{a}$, $a\in\bR^2$ is an open continuous mapping
    from $\bR^2$ to $\coinv{0}{+\infty}$
    and $\mathbf{C}=\varphi[F]$, the set $F$ must be nowhere dense.
    Hence, $F$ is not a regular closed set.
    
    We will show that $F$ is a local Lipschitz one set.
    By the construction from Example \ref{ex:FCS}, $\mathbf{C}$ is quasi-dense.
    Then by Theorem \ref{thm:q_dense_iff_monot_LLip_one},
    there exits a function $g\colon\bR\to\bR$ such that
    $\LLip g=\charfunction{\mathbf{C}}$. Define $f\colon\bR^2\to\bR$
    by letting $f(a)=g(\norm{a})$, $a\in\bR^2$.
    It remains to show that $\LLip f=\charfunction{F}$.
    Let $a_0\in\bR^2\setminus F$. Then $t_0:=\norm{a_0}\notin\mathbf{C}$.
    Fix $\varepsilon>0$. Since $\LLip g(t_0)=0<\varepsilon$, there exists
    $\delta>0$ such that for all $s,t\in\oinv{t_0-\delta}{t_0+\delta}$,
    \[
        \abs{g(s)-g(t)}\leq\varepsilon\abs{s-t}.
    \]
    The set $U_0=\varphi^{-1}\big[\oinv{t_0-\delta}{t_0+\delta}\big]$
    is a neighborhood of $a_0$. Consider $p,q\in U_0$ and let
    $s=\norm{p}$, $t=\norm{q}$. Then, $s,t\in\oinv{t_0-\delta}{t_0+\delta}$
    and so,
    \begin{align*}
        \norm{f(p)-f(q)} &= \abs{g(s)-g(t)} \leq \varepsilon\abs{s-t} \\
                         &= \varepsilon\big\vert\norm{p}-\norm{q}\big\vert \leq \varepsilon\norm{p-q}.
    \end{align*}
    That is, $\LLip f(a_0)\leq\varepsilon$.
    With $\varepsilon\to 0$ we obtain the equality $\LLip f(a_0)=0$.
    Note that $g$ is a $1$-Lipschitz function, so also $f$ is $1$-Lipschitz.
    Indeed, for any $x,y\in\bR$,
    \begin{align*}
        \norm{f(x)-f(y)} &= \abs{g(\norm{x})-g(\norm{y})} \\
                         &\leq \varepsilon\big\vert\norm{x}-\norm{y}\big\vert \leq \varepsilon\norm{x-y}.
    \end{align*}
    In particular, $\LLip f\leq 1$. Fix $a_0\in F$.
    We will show that $\LLip f(a_0)\geq 1$. Since $t_0:=\norm{a_0}\in F$,
    we have $\LLip g(t_0)=1$. Let $U_0$ be an open neighborhood of $a_0$ in $\coinv{0}{+\infty}$.
    Since $\varphi$ is an open mapping, $V_0=\varphi[U_0]$ is an open
    neighborhood of $t_0$. Hence, $\big\Vert g\vert_{V_0}\big\Vert_{\mathrm{lip}}\geq\LLip g(t_0)=1$.
    Fix $\alpha<1$. Then 
    \[
        \lipnorm{g\vert_{V_0}}=\sup\limits_{s,t\in V_0}\frac{\abs{g(s)-g(t)}}{\abs{s-t}}>\alpha.
    \]
    So, there exists $s,t\in V_0$ such that $\frac{\abs{g(s)-g(t)}}{\abs{s-t}}>\alpha$.
    Hence, we see that
    \[
        \LLip f(a_0)\geq\lipnorm{f\vert_{U_0}}\geq\frac{\abs{g(s)-g(t)}}{\abs{s-t}}>\alpha.
    \]
    With $\alpha\to 1$ we obtain the inequality $\LLip f(a_0)\geq 1$.
    We have proven that $\LLip f=\charfunction{F}$.
    \qed
\end{exampl}

\section*{Acknowledgments}
The author would like to thank his supervisor, Oleksandr Maslyuchenko,
for his constant support and many useful suggestions.

\end{document}